\documentclass{amsart}

\usepackage{color}
\usepackage{amssymb}
\usepackage{calrsfs}
\usepackage{hyperref}

\numberwithin{equation}{section}

\newtheorem{Thm}{Theorem}[section]
\newtheorem{Prop}[Thm]{Proposition}
\newtheorem{Lem}[Thm]{Lemma}
\newtheorem{Cor}[Thm]{Corollary}
\theoremstyle{definition}

\newtheorem{Rem}[Thm]{Remark}

\newcommand{\diam}{\operatorname{diam}}
\newcommand{\dist}{\operatorname{dist}}

\newcommand{\N}{\mathbb{N}}
\newcommand{\R}{\mathbb{R}}

\newcommand{\cC}{\mathcal{C}}
\newcommand{\cD}{\mathcal{D}}

\newcommand{\cH}{\mathcal{H}}

\newcommand{\defl}{\mathrel{\mathop:}=}

\title{Box-counting by H\"older's traveling salesman}

\author{Zolt\'an M. Balogh}

\author{Roger Z\"{u}st}

\date{\today}

\begin{document}

\begin{abstract}
We provide a sufficient Dini-type condition for a subset of a complete, quasiconvex metric space to be covered by a H\"older curve. This implies in particular that if the upper box-counting dimension of a set in a quasiconvex metric space is less or equal to $d \geq 1$, then for any $\alpha < \frac{1}{d}$ the set can be covered by an $\alpha$-H\"older curve. On the other hand, for each $1\leq d <2$ we give an example of a compact set $K$, in the plane, just failing the above Dini-type condition, with lower box-counting dimension equal to zero and upper box-counting dimension equal to $d$ that can not be covered by a countable collection of $\frac{1}{d}$-H\"older curves.
\end{abstract}


\thanks{This research was partially supported by the Swiss National Science Foundation.}
\maketitle


\section{Introduction}

One way to measure the size of a set in a metric space is its dimension. The concepts of Hausdorff and box-counting dimension are especially relevant in this respect. We refer to the book of Falconer \cite{F} for an excellent overview of the subject. On the other hand, we can think of a planar set  being large, if the points of the set cannot be visited in finite time by a salesman traveling with bounded speed. These sets are unrectifiable in the sense that  they cannot be covered by a rectifiable curve. More generally, one could study the size of a set in terms of its property to be covered by a H\"older continuous curve. 

Jones gave a necessary and sufficient condition for a bounded planar set $K$ to be covered by a rectifiable curve in terms of the so called $\beta$-number $\beta(K)$ introduced in \cite{J}. To recall the definition of $\beta(K)$ let us consider first the local $\beta$-number within a square $Q$: $\beta_K(Q) \defl \frac{\omega(Q)}{\ell(Q)}$, where $\omega(Q)$ is the smallest width of a line strip that covers $Q \cap K$. Now the beta-number of $K$ is
\[
\beta^2(K) \defl \sum_Q \beta_K^2(3Q)\ell(Q) \ ,
\]
where the sum is over all dyadic squares in $\R^2$ and $3Q$ is the axis-parallel square with the same center as $Q$ but side length $3\ell(Q)$. According to the main result of \cite{J} a bounded planar set $K$ can be covered by a curve of finite length if and only if $\beta(K) < \infty$. Generalizations of this result in higher dimensions is due to Okikiolu \cite{O} and to Hilbert spaces to Schul \cite{S}.

For connected sets in the plane Bishop and Jones \cite{B-J} proved that a lower bound on the local $\beta$ numbers $\beta_K(Q)$ of a compact set $K\subset \R^2$ at all scales implies that the Hausdorff dimension of $K$ is strictly larger than $1$. A generalization of this result to the setting of metric spaces is due to Azzam \cite{A}.

Let us observe first that for disconnected sets the situation is very different. To see this consider the standard $\frac{1}{4}$ four corner Cantor set $C$ constructed as follows. The set $C$ is the intersection of compact sets $C_k$, where $C_0$ is the unit square $[0,1] \times [0,1]$ and for $k\geq 1$ the set $C_k$ is composed by a number of $4^{k}$ squares of side-length $\frac{1}{4^{k}}$. The squares in $C_{k}$ are obtained by replacing each square $Q$ in $C_{k-1}$ by its four corner squares of side-lengths $\frac{1}{4}$ of the side-length of $Q$. It is easy to check that if $Q$ is any square with side-length less than $1$ such that $Q \cap C \neq \emptyset$, then $\beta_C(3Q) \geq c$ for some $c>0$ that does not depend on $Q$. The fact that the Hausdorff dimension of $C$ is equal to $1$ is in contrast to the results of \cite{A} and \cite{B-J}. 

A naturally related question is to consider H\"older curves instead of Lipschitz ones. The problem has been recently studied by Badger and Vellis \cite{B-V} and Badger-Naples-Vellis in \cite{B-N-V}. In this note we consider the relationship  between the property of a set to be covered by a H\"older curve and box-counting dimension. To formulate our results we start with some notation. 


Let $S$ be a metric space. The $t$-dimensional Hausdorff measure of $S$ for $t \geq 0$ is defined by $\cH^t(S) \defl \lim_{\delta \downarrow 0} \cH^t_\delta(S)$, where $\cH^t_\delta(S)$ is defined for $\delta > 0$ by
\[
\cH^t_\delta(S) \defl \left\{\sum_{i=1}^\infty \diam(S_i)^t : S \subset \bigcup_{i=1}^\infty S_i, \diam(S_i) \leq \delta \right\} \ .
\]
The Hausdorff dimension is defined to be $\dim_{\rm H}(S) \defl \inf\{t\geq 0 : \cH^t(S)=0\}$. The upper and lower box-counting dimension of $S$ are defined by
\[
\underline{\dim}_{\rm box}(S) \defl \liminf_{\epsilon \downarrow 0} \frac{\log N(S,\epsilon)}{-\log \epsilon} \quad \mbox{and} \quad \overline{\dim}_{\rm box}(S) \defl \limsup_{\epsilon \downarrow 0} \frac{\log N(S,\epsilon)}{-\log \epsilon} \ ,
\]
where $N(S,\epsilon)$ is the minimal number of balls of radius $\epsilon$ needed to cover $S$. Note that in case $S \subset \R^2$ and $N'(S,\epsilon)$ is the minimal number of squares with edge length $\epsilon$ needed to cover $S$, then $C^{-1}N(S,\epsilon) \leq N'(S,\epsilon) \leq CN(S,\epsilon)$ for some constant $C \geq 1$ independent of $S$. This shows that in the above definitions we can replace $N(S,\epsilon)$ by $N'(S,\epsilon)$. For our discussion it also doesn't matter if we assume the centers of balls used to cover $S$ to be contained in $S$ or not. We refer to the book of Falconer \cite{F} for more details. 


The following simple result provides information on the behavior of box-counting dimension under a H\"older map. 

\begin{Prop}
	\label{holder_prop}
Let $0 < \alpha \leq 1$ and $f : X \to Y$ be surjective $\alpha$-H\"older map between metric spaces. Then $N(Y,L\epsilon^\alpha) \leq N(X,\epsilon)$ for some $L > 0$ and all $\epsilon > 0$. In particular, $\underline{\dim}_{\rm box}(Y) \leq \frac{1}{\alpha} \cdot \underline{\dim}_{\rm box}(X)$ and $\overline{\dim}_{\rm box}(Y) \leq \frac{1}{\alpha} \cdot \overline{\dim}_{\rm box}(X)$.
\end{Prop}

The easy proof is left as an exercise to the reader. Taking $X = [0,1]$ with the usual metric, an obvious necessary condition that a metric space $Y$ can be covered by an $\alpha$-H\"older path $f : [0,1] \to Y$ is that $N(Y,\epsilon)\epsilon^\frac{1}{\alpha} \leq C$ for some $C \geq 0$ and all $0 < \epsilon \leq 1$. The following result gives a related sufficient condition for covering a set in a complete, quasiconvex metric space by a H\"older curve. Recall that a metric space $X$ is called quasiconvex if there is a constant $C_X \geq 1$ such that any two points $x,y \in X$ can be connected by a path $\gamma : [0,1] \to X$ of length $\ell(\gamma) \leq C_Xd(x,y)$.

\begin{Thm}
	\label{cover_prop}
Let $Y$ be a subset of a complete, quasiconvex metric space $X$ and assume that $N(Y,\epsilon)\epsilon^d$ satisfies a Dini condition for some $d \geq 1$, namely,
\begin{equation} \label{dini_cond}
\sum_{k \geq 0} N(Y,\epsilon_0 2^{-k})2^{-kd} < \infty
\end{equation} 
for some $\epsilon_0 > 0$. Then $Y$ can be covered by a $\frac{1}{d}$-H\"older curve in $X$.
\end{Thm}

An immediate consequence of this is the following:

\begin{Cor}
	\label{cover_cor}
Let $X$ be a complete, quasiconvex metric space and assume that $Y \subset X$ satisfies $\overline{\dim}_{\rm box}(Y) < d$. Then $Y$ can be covered by a $\frac{1}{d}$-H\"older curve.
\end{Cor}

To prove Corollary~\ref{cover_cor} we need to check that condition \eqref{dini_cond} is satisfied if $\overline{\dim}_{\rm box}(Y) < d$. Indeed, in this case there exists a $\delta >0$ such that 
\[
\log(N(Y,\epsilon)) \leq (d-\delta)\cdot \log\frac{1}{\epsilon}
\]
for all $\epsilon$ small enough. This implies that  $N(Y,\epsilon)\epsilon^d \leq C\epsilon^{\delta}$ for some fixed  constant $C > 0$ independently of $\epsilon >0$.  Hence $N(Y,\epsilon)\epsilon^d $ satisfies indeed the Dini condition in the statement of the theorem above.

Let us remark that, although our result works in the general metric setting, in the Euclidean space $\R^n$ a much stronger result is available due to Badger-Naples-Vellis \cite{B-N-V}. Here the authors proved that the condition 
\begin{equation}
	\label{badger_cond}
\sum_{k \geq 0} \tilde{N} (Y, 2^{-k})2^{-kd} < \infty \ , 
\end{equation}
is sufficient for $Y \subset \R^n$ to be covered by an $\frac{1}{d}$-H\"older curve. In the above relation we use the notation 
\[
\tilde{N} (Y, 2^{-k}) = \sharp \{Q \in \cD_k: \beta_Y(Q)\geq  \beta_0 \} \ ,
\]
where $\cD_k$ is the collection of dyadic cubes in $\R^n$ with side length $2^{-k}$ and $\beta_0>0$ is a fixed constant. 

The second result of this note is an example showing that a bound on the lower box-counting dimension does not imply the H\"older covering property. We prove the existence of a small set of vanishing Hausdorff dimension that cannot be covered by a countable union of H\"older curves. This is formulated in the following:

\begin{Thm}
	\label{main_thm}
For any $1 \leq d < 2$ there is a compact set $K^d \subset \R^2$ such that $\dim_{\rm H}(K^d) = \underline{\dim}_{\rm box}(K^d) = 0$ and
\begin{equation} \label{lim_cond}
\lim_{\epsilon \downarrow 0} N(K^d,\epsilon)\epsilon^d  = 0 
\end{equation}
but $K^d$ can not be covered by a countable collection of $\frac{1}{d}$-H\"older curves.
\end{Thm}

Note that the restriction $d < 2$ is necessary because of the existence of Peano curves in $\R^2$ that are H\"older continuous with exponent $\frac{1}{2}$. The construction of $K^d$ above can naturally be extended to $\R^n$, where the same statement is true with $1 \leq d < n$. The properties of $K^d$ in the statement of the theorem in particular imply that $\overline{\dim}_{\rm box}(K^d) = d$. Indeed, it is easy to check that $N(K^d,\epsilon) \leq C \epsilon^{-d}$ for all small $\epsilon > 0$ implies that $\overline{\dim}_{\rm box}(K^d) \leq d$. The lower bound $\overline{\dim}_{\rm box}(K^d) \geq d$ follows from Corollary~\ref{cover_cor}.

Applying the above theorem for a sequence of $d_n \to 2$ and taking the union of appropriately scaled and translated copies of $K^{d_n}$ we obtain the following corollary:

\begin{Cor}
	\label{holder_thm}
There is a compact set $K \subset \R^2$ with $\dim_{\rm H}(K) = 0$ that can not be covered by a countable collection of $\alpha$-H\"older curves for any $\alpha > \frac{1}{2}$.
\end{Cor}

As above we observe that necessarily $\overline{\dim}_{\rm box}(K) =2$. Note that since the lower box-counting dimension is not stable under unions of sets we cannot conclude that for this example $\underline{\dim}_{\rm box}(K) = 0$. However, a modified version of the construction provided in the proof of Theorem~\ref{main_thm} could give also this property. 

Theorem~\ref{main_thm} also serves as a counterexample for two definitions of rectifiability. Following the notation of Federer \cite[3.2.14]{Fed} a subset $S$ of a metric space $X$ is countably $m$-rectifiable if there exist countably many Lipschitz maps $f_i : E_i \to X$, $i \in \N$ defined on bounded sets $E_i \subset \R^m$ such that $S = \bigcup_{i \in \N} f_i(E_i)$. More generally, $S$ is countably $(\cH^m,m)$-rectifiable if $S$ is $\cH^m$-measurable and it can be expressed as a union $S = S' \cup A$ of a countably $m$-rectifiable set $S'$ and a set $A \subset X$ with $\cH^m(A) = 0$. At first glance one might think that the set $A$ in the second definition is small enough so that it can be easily covered by a countably $m$-rectifiable set too. Already in the case $d=1$, Theorem~\ref{main_thm} demonstrates that the two definitions above are different and $K^1$ although it is countably $(\cH^1,1)$-rectifiable because $\dim_{\rm H}(K^1) = 0$, it is not countably $1$-rectifiable. The existence of such sets are also guaranteed by \cite{K} as level sets of a $C^1$-map $f: [0,1]^2 \to [0,1]$.

\section{Proof of Theorem~\ref{cover_prop}}

\begin{proof}[Proof of Theorem~\ref{cover_prop}]
First note that the covering condition on $Y$ guarantees that $Y$ is totally bounded, i.e.\ for any $r > 0$ there are finitely many balls of radius $r$ that cover $Y$. Thus since $X$ is complete, the closure $\bar Y$ of $Y$ is compact.
	
For each $k \geq 0$ let $\cC_k$ be a minimal cover by closed balls with radius $\epsilon_0 2^{-k}$. Without loss of generality we assume that $\cC_0$ consists of a single ball. Otherwise we start our construction by connecting the centers of balls in $\cC_0$ by a single closed curve of finite length.

For any $k \geq 1$ and any ball $B \in \cC_k$ choose a parent ball $B' \in \cC_{k-1}$ with the property that $B' \cap B \neq \emptyset$. Such a ball exists because of the minimality of the cover $\cC_k$. Note that by this procedure we obtain for each ball a unique parent, but a given ball can have several or no children.

Let us recall that a metric tree is by definition a geodesic metric space that does not contain embedded circles. We construct a complete metric tree $T$ of finite length as follows: $T$ is obtained as a Gromov-Hausdorff limit of finite metric trees $T_k$ in which the vertices correspond to balls in $\bigcup_{0 \leq l \leq k} \cC_l$. First, $T_1$ consists of $\sharp \cC_1$ closed intervals of length $2^{-d}$ glued together at an endpoint equipped with the quotient length metric. The other endpoints of these intervalls form vertices of $T_1$ that are in one-to-one correspondence with balls in $\cC_1$. Iteratively, $T_{k+1}$ is constructed from $T_k$ by attaching to a vertex in $T_k$ that corresponds to some $B \in \cC_k$ as many intervals of length $2^{-kd}$ as $B$ has children in $\cC_{k+1}$. Due to the fact that $\sum_{k \geq 1}(\sharp \cC_k) 2^{-kd} < \infty$, the sequence $(T_k)_{k \geq 1}$ of compact metric trees has uniformly bounded total length. Thus this sequence has a limit $T$ with respect to the Gromov-Hausdorff distance by Gromov's compactness theorem, see e.g.\ \cite[Theorem~7.4.15]{B-B-I}. As such a limit $T$ is itself a compact metric tree with bounded total length. It is clear that $T$ contains an isometric copy of each $T_k$ in an obvious way.

In the next step, we successively define for each $m\geq 1$ maps $\varphi : T_m \to X$ as follows: First we define the $\varphi$ on the vertices of $T_m$. If $y_B \in T_m$ is the vertex that corresponds to $B \in \bigcup_{0 \leq k \leq m} \cC_k$, then $\varphi(y_B)$ is defined as the center $c_B$ of $B$. Let us investigate the metric distortion properties of $\varphi$. Consider two vertices $y_{B},y_C \in T_k$ such that $B \in \cC_k$, $C \in \cC_l$ and $C$ is a descendant of $B$ (so $l < k \leq m$). Then
\begin{align*}
d_X(\varphi(y_B),\varphi(y_C)) & = d_X(c_B,c_C) \leq \epsilon_0\sum_{i=l}^{k-1} 2^{-i} + 2^{-i-1} \leq 2\epsilon_0\sum_{i=l}^{k} 2^{-i} \\
 & \leq 4\epsilon_02^{-l} = 4\epsilon_0(2^{-ld})^\frac{1}{d} \leq 4\epsilon_0 d_T(y_B,y_C)^\frac{1}{d} \ .
\end{align*}
In the first inequality we used that
\[
d_X(c_{B_{i}},c_{B_{i+1}}) \leq \epsilon_02^{-i} + \epsilon_02^{-i-1}
\]
if $B_i \in \cC_i$ is the parent of $B_{i+1} \in \cC_{i+1}$ because these balls have nonempty intersection. For the last inequality note that $d_T(y_B,y_C) = \sum_{i = l}^{k-1} 2^{-id}$ by the construction of $T$. Now assume that $C$ is not necessarily a descendant of $B$. If $A \in \cC_n$ the least common ancestor of $B$ and $C$, the geodesic connecting $y_B$ with $y_C$ in $T_m$ goes through $y_A$ and therefore
\begin{align*}
d_X(\varphi(y_B),\varphi(y_C)) & = d_X(c_B,c_C) \leq d_X(c_B,c_A) + d_X(c_A,c_C) \\
 & \leq 4\epsilon_0 \bigl(d_T(y_B,y_A)^\frac{1}{d} + d_T(y_A,y_C)^\frac{1}{d}\bigr) \\
 & \leq 8 \epsilon_0 d_T(y_B,y_C)^\frac{1}{d} \ .
\end{align*}
In the second step, the map $\varphi$ is linearly extended to the segments of $T_m$: Let $B'$ be the parent of $B$ corresponding to vertices in $T_m$ and let $\gamma_T : [0,1] \to T_m$ and $\gamma_X : [0,1] \to X$ be paths connecting $y_{B'}$ with $y_B$ and $c_{B'}$  with $c_B$ respectively such that $d_T(\gamma_T(a),\gamma_T(b)) = d_T(y_{B'},y_B)|a-b|$ and $d_X(\gamma_X(a),\gamma_X(b)) \leq C_X d_X(c_{B'},c_B)|a-b|$, where $C_X$ is the constant of quasiconvexity of $X$. Now on the segment connecting $y_{B'}$ with $y_B$ in $T$ the map $\varphi$ is defined by $\varphi(\gamma_T(t)) \defl \gamma_X(t)$. Similar to the estimates above one checks that with this extension we obtain a $\frac{1}{d}$-H\"older map $\varphi : T_m \to X$ with a uniform bound on the H\"older constant. 

Starting with $T_1$ and extending this map successively to $T_m$ for $m > 1$ we obtain a H\"older map $\varphi : T' \defl \bigcup_{m \geq 1} T_m \subset T \to X$ of the same regularity. Because $T'$ is dense in $T$, $\varphi$ can be extended to a map on $T$ of the same regularity. Since the image $\varphi(T)$ contains in particular all the centers of balls in $\cC_k$, it is clear that $Y$ is contained in the closure of $\varphi(T)$. Because $T$ is compact this implies $\bar Y \subset \varphi(T)$ and concludes the proof.
\end{proof}

\section{Proof of Theorem~\ref{main_thm}}

The general idea of the proof of Theorem \ref{main_thm} is based on a modification of the construction of the standard four corner $\frac{1}{4}$ Cantor set $C$. It is well known that $C$ is not rectifiable and  $\dim_{\rm H}(C) = \underline{\dim}_{\rm box}(C) =\overline{\dim}_{\rm box}(C) = 1$. We shall modify the construction of $C$ by pushing down the lower box-counting dimension to $0$ while still retaining the unrectifiability property. This is achieved by an iterative construction on two different alternating collections of scales. For all scales in the first collection the relative size is kept large to guarantee the unrectifiability property. For the second collection of scales the relative size is drastically reduced to achieve that $\dim_{\rm H}(K) =0$. The proof is carried out in two steps: in the first step we give the construction for $d=1$; in the second step we construct $K^d$ for general $d$. In this section we assume that $\R^2$ is equipped with the sup-norm, so that the side length of squares agrees with their diameter and in the definition of $N(K,\epsilon)$ we use cover by squares with side length $\epsilon$ instead of balls with radius $\epsilon$.

\subsection{Construction of $K^1$}

We first describe the construction of $K \defl K^1$ in Theorem~\ref{main_thm} depending on a parameter $\frac{1}{2} < \gamma < 1$ and later we modify this to obtain $K^d$. The compact set $K \subset \R^2$ is obtained as $K = \bigcap_{k \geq 0} K_{k}$ where $K_{k}$ is the union of a collection $\cC_{k}$ of $4^k$ disjoint closed squares that are constructed iteratively. Each square in $\cC_{k}$ has side length $\ell_{k}$. In order to describe $K_{k}$ consider a strictly increasing sequence $(k_n)_{n \geq 0}$ of nonnegative integers with $k_0 = 0$. This sequence will be determined later but note that it does not depend on $d$. First $\cC_{0} = \{K_0\}$ where $K_0 \defl [0,1]^2$ and hence $\ell_0 = 1$. In case $k_{2n} \leq k < k_{2n+1}$ for some $n \geq 0$, then each square in $\cC_{k}$ is replaced by the $4$ corner-squares with side length $\ell_{k+1} = \frac{1}{4^\gamma} \ell_{k}$. Note that these squares are disjoint because $\frac{1}{4^\gamma} < \frac{1}{2}$ by the restriction on $\gamma$. These new squares compose $\cC_{k+1}$. In case $k_{2n-1} \leq k < k_{2n}$ for some $n \geq 1$, then each square in $\cC_{k}$ is replaced by the $4$ corner squares of side length $\ell_{k+1} = \frac{1}{4^n} \ell_{k}$. In the following Lemma are estimates of $\ell_{k}$.

\begin{Lem}
	\label{sidelength_lem}
For $n \geq 1$ it holds that
\begin{alignat*}{2}
	4^{(\gamma-n+1) k_{2n-2}+(n-\gamma) k_{2n-1}-nk_{2n}} & \leq \ell_{k_{2n}} && \leq 4^{(n-\gamma) k_{2n-1} - nk_{2n}} \ , \\
	4^{(n-\gamma)k_{2n-1} + (\gamma - n)k_{2n}-\gamma k_{2n+1}} & \geq \ell_{k_{2n+1}} && \geq 4^{(\gamma -n)k_{2n}-\gamma k_{2n+1}} \ .
\end{alignat*}
\end{Lem}

\begin{proof}
In order to estimate $\ell_{k_{2n}}$ we note that in each step from $k_{2n-1}$ to $k_{2n}$ the length of the squares get multiplied by a factor $\frac{1}{4^n}$ and in each step before $k_{2n-1}$ the length gets multiplied by a factor of at most $\frac{1}{4^\gamma}$. Thus
\[
\ell_{k_{2n}} \leq 4^{-\gamma k_{2n-1}} 4^{-n(k_{2n}-k_{2n-1})} = 4^{(n-\gamma)k_{2n-1}-nk_{2n}} \ .
\]
To obtain the upper bound for $\ell_{k_{2n+1}}$ note that
\begin{align*}
\ell_{k_{2n+1}} & = \ell_{k_{2n}}4^{-\gamma(k_{2n+1}-k_{2n})} \leq 4^{(n-\gamma)k_{2n-1} + (\gamma - n)k_{2n}-\gamma k_{2n+1}} \ .
\end{align*}
Similarly to the upper bound for $\ell_{k_{2n}}$ we obtain the lower bound
\[
\ell_{k_{2n+1}} \geq  4^{-n k_{2n}} 4^{-\gamma(k_{2n+1}-k_{2n})} = 4^{(\gamma-n) k_{2n}-\gamma k_{2n+1}} \ .
\]
This implies the following lower bound for $\ell_{k_{2n}}$
\begin{align*}
\ell_{k_{2n}} & \geq \ell_{k_{2n-1}}4^{-n(k_{2n}-k_{2n-1})} = 4^{(\gamma-n+1) k_{2n-2}+(n-\gamma) k_{2n-1}-nk_{2n}} \ .
\end{align*}
This last estimate also holds for $n=1$ since $k_0=0$.
\end{proof}

We define $k_0=0$, $k_1=1$ and
\begin{align*}
k_{2n} & \defl \left\lceil\frac{n}{1-\gamma}k_{2n-1}\right\rceil \ , \\
k_{2n+1} & \defl \left\lceil \frac{1}{1-\gamma}\left( \log_4(n) + \sum_{i=1}^{2n}(-1)^{i} (\lceil \tfrac{i}{2}\rceil-\gamma) k_i\right)\right\rceil \ ,
\end{align*}
for $n \geq 1$.  Here the notation $\lceil x\rceil$ stands for the smallest integer that is greater or equal to $x$.

\medskip

Note that $k_{2n} > k_{2n-1}$ since $\gamma > 0$. We have also $k_{2n+1} > k_{2n}$. This is because
\begin{align*}
(1-\gamma)k_{2n+1} & \geq ((\lceil \tfrac{2n}{2}\rceil-\gamma)k_{2n}-(\lceil \tfrac{2n-1}{2}\rceil-\gamma)k_{2n-1}) + \cdots \\ 
 & \qquad + ((1-\gamma)k_{2}-(1-\gamma)k_{1}) \\
 & \geq (n-\gamma)(k_{2n}-k_{2n-1}) \\
 & \geq (n-\gamma)k_{2n} - (1-\gamma)k_{2n} \\
 & = (n-1)k_{2n} \ .
\end{align*}
Above we used that $(1-\gamma)k_{2n} \geq (n-\gamma)k_{2n-1}$. Therefore $k_{n+1} \geq cnk_n$ for some $c > 0$ and all $n \geq 1$. 
Hence
\[
1-\gamma = (n-\gamma)\frac{k_{2n}}{k_{2n+1}} + \epsilon_n \ ,
\]
where
\[
|\epsilon_n| \lesssim \frac{nk_{2n-1}}{k_{2n+1}} \leq \frac{n}{c^2(2n-1)2n} \lesssim \frac{1}{n} \ .
\]
Thus
\begin{equation}
\label{sequence_def}
\lim_{n\to\infty} \frac{n}{2}\frac{k_n}{k_{n+1}} = 1-\gamma \ .
\end{equation}
We now give a precise estimate for $\ell_{2n+1}$. For $n \geq 1$ it holds
\begin{align*}
\ell_{k_{2n+1}} & = \ell_{k_{2n-1}}4^{-n(k_{2n}-k_{2n-1})}4^{-\gamma(k_{2n+1}-k_{2n})} \\
 & = \ell_{k_{2n-1}}4^{nk_{2n-1} -(n-\gamma) k_{2n} - \gamma k_{2n+1}} \ .
\end{align*}
Because $\ell_{k_{1}} = 4^{-\gamma}\ell_{k_0} = 4^{-\gamma}$ and thus $\ell_{k_3} = 4^{(1-\gamma) - (1-\gamma)k_2 - \gamma k_3}$, it follows
\begin{align}
\nonumber
\ell_{k_{2n+1}} & = 4^{- \gamma k_{2n+1} - \sum_{i=1}^{2n} (-1)^i (\lceil\frac{i}{2}\rceil-\gamma)k_i} \\
\nonumber
 & \asymp 4^{- \gamma k_{2n+1} - \log_4(n) - (1-\gamma)k_{2n+1}} \\
\label{preciselength_eq}
 & = \tfrac{1}{n}4^{-k_{2n+1}} \ .
\end{align}
In this section we use the notation $a \lesssim b$ for $a,b \geq 0$ to mean that $b \leq ca$ for some $c \geq 0$ depending only on $\gamma$. Similarly, $a \asymp b$ means that $a \lesssim b$ and $b \lesssim a$.

\begin{Lem}
	\label{uperbox_lem}
The estimate $N(K,\epsilon) \leq C(\epsilon) \frac{1}{\epsilon}$ holds for some $C(\epsilon) \geq 0$ such that $\lim_{\epsilon \downarrow 0}C(\epsilon) = 0$.
\end{Lem}

\begin{proof}
Note first that for any $k$ it holds that $N(K,\ell_k) = 4^k$: First $N(K,\ell_k) \leq 4^k$ is clear because of the obvious cover $\cC_k$. On the other hand, the corners of any square in $\cC_k$ belong to $K$ and there are $4^{k+1}$ such corners. A square in $\R^2$ with side length $\ell_k$ can cover at most $4$ such corners and therefore $N(K,\ell_k) \geq 4^k$.

We need to compare $\ell_k$ with $4^k$. If $k_{2n} \leq k \leq k_{2n+1}$ (and $k \neq 0$), then it follows from \eqref{preciselength_eq}
\begin{align*}
N(K,\ell_{k})\ell_{k} & = 4^k \ell_{k_{2n+1}}4^{\gamma(k_{2n+1}-k)} = 4^{k_{2n+1}}\ell_{k_{2n+1}} 4^{k-k_{2n+1} + \gamma(k_{2n+1}-k)} \\
& = 4^{k_{2n+1}}\ell_{k_{2n+1}} 4^{(\gamma-1)(k_{2n+1}-k)} \\
& \lesssim \tfrac{1}{n} \ .
\end{align*}
Now fix $k_{2n-1} \leq k \leq k_{2n}$. Similarly to the estimate above
\begin{align*}
N(K,\ell_{k})\ell_{k} & = 4^k \ell_{k_{2n-1}}4^{-n(k-k_{2n-1})} = 4^{k_{2n-1}}\ell_{k_{2n-1}}4^{k-k_{2n-1} - n(k-k_{2n-1})} \\
 & \leq 4^{k_{2n-1}}\ell_{k_{2n-1}} \\
 & \lesssim \tfrac{1}{n-1} \ .
\end{align*}
As a consequence we note that  $N(K,\ell_{k})\ell_{k} \to 0$ as $k \to \infty$.

Next consider a general $\epsilon >0$ small. 
Assume that $k$ is such that $\ell_{k+1} \leq \epsilon \leq \ell_{k}$. Clearly $\epsilon \to 0$ implies that $k \to \infty$. 

It holds $N(K,\epsilon) \leq N(K,\ell_{k+1}) = 4N(K,\ell_{k})$ and therefore
\begin{align*}
N(K,\epsilon) \epsilon & \leq 4N(K,\ell_{k}) \ell_{k} \frac{\epsilon}{\ell_{k}} \leq 4N(K,\ell_{k}) \ell_{k}
\end{align*}
Thus $\lim_{\epsilon \downarrow 0} N(K,\epsilon) \epsilon = 0$ and this concludes the proof.
\end{proof}

\begin{Lem}
	\label{lowerbox_lem}
$\dim_{\rm H}(K) = \underline{\dim}_{\rm box}(K) = 0$.
\end{Lem}

\begin{proof}
To prove the lemma we shall use the scales $\ell_{k_{2n}}$. Clearly, there are $4^{k_{2n}}$ squares in $\cC_{k_{2n}}$ with diameter $\ell_{k_{2n}}$. Thus $N(K,\ell_{k_{2n}}) \leq 4^{k_{2n}}$. With the estimates in Lemma~\ref{sidelength_lem} we obtain $\ell_{k_{2n}} \leq 4^{(n-\gamma) k_{2n-1} - nk_{2n}}$. Thus
\begin{align*}
\frac{\log_4 N(K,\ell_{k_{2n}})}{\log_4 \ell_{k_{2n}}^{-1}} \leq \frac{k_{2n}}{k_{2n}(n - (n-\gamma) k_{2n-1}/k_{2n})} = \frac{1}{n - (n-\gamma) k_{2n-1}/k_{2n}} \ .
\end{align*}
Because $(n-\gamma) k_{2n-1}/k_{2n}$ is bounded in $n$ by \eqref{sequence_def} it follows that
\[
\lim_{n \to \infty} \frac{\log_4 N(K,\ell_{k_{2n}})}{\log_4 \ell_{k_{2n}}^{-1}} = 0
\]
and thus $\underline{\dim}_{\rm box}(K) = 0$. It is a general fact that $\dim_{\rm H}(K) \leq \underline{\dim}_{\rm box}(K)$ and hence also $\dim_{\rm H}(K) = 0$. But this can also be verified by a direct computation: For any $t > 0$ and all $n \geq 1$
\begin{align*}
\cH^t_{\ell_{k_{2n}}}(K) & \leq \cH^t_{\ell_{k_{2n}}}(K_{k_{2n}}) \leq 4^{k_{2n}}4^{tk_{2n}((n-\gamma) k_{2n-1}/k_{2n} - n)} \\
& = 4^{k_{2n}(1 + t(n-\gamma) k_{2n-1}/k_{2n} - tn)} \ ,
\end{align*}
and this converges to zero  because the exponent goes to $-\infty$ for $n \to \infty$.
\end{proof}

\begin{Rem}
Using the scales the scales $\ell_{k_{2n+1}}$ one can show that $\overline{\dim}_{\rm box}(K) \geq 1$. By Lemma \ref{uperbox_lem} we obtain that in fact $\overline{\dim}_{\rm box}(K) = 1$.
\end{Rem}

In order to show that $K$ can not be covered by countably many rectifiable curves, we shall use  the Theorem of Jones \cite{J}.

\begin{Lem}
	\label{betainfinity_lem}
$\beta^2(K) = \infty$, and therefore $K$ can not be covered by a curve $\gamma : [0,1] \to \R^2$ of finite length.
\end{Lem}

\begin{proof}
Let $n \geq 1$ and set $k'_n \defl k_{2n+1}-1$. Le $i_n$ be the unique integer such that $2^{-i_n} \leq \ell_{k'_n} < 2^{-i_n+1}$. Because $\ell_{k_{2n+1}} = 4^{-\gamma}\ell_{k'_n}$ and $\frac{1}{4} < 4^{-\gamma} < \frac{1}{2}$ it holds that
\begin{equation}
	\label{dyadic_eq}
2^{-i_n-2} \leq \ell_{k_{2n+1}} \leq 2^{-i_n} \ .
\end{equation}
Let $A_n$ be the corner set of squares in $\cC_{k'_{n}}$. It holds: $A_n \subset K$, $\sharp A = 4^{k'_n + 1}$ and $|a-b| \geq \ell_{k'_n} \geq 2^{-i_n}$ for different corners $a,b \in A_n$ of the same square in $\cC_{k'_n}$. Let us denote by $\cD_{i_n}$ the collection of dyadic squares $Q$ in $\R^2$ with side length $2^{-i_n}$ such that $Q \cap A_n \neq \emptyset$. By the reasoning as in the proof of Lemma~\ref{uperbox_lem} a square in $\cD_{i_n}$ can cover at most $4$ points of $A_n$ and thus
\begin{equation}
	\label{numbersquaresestimate_eq}
	\sharp\cD_{i_n} \geq 4^{k'_n} = 4^{k_{2n+1}-1} \ .
\end{equation}
Because of \eqref{dyadic_eq}, for any square $Q \in \cD_{i_n}$, the square $3Q$ contains some $Q' \in \cC_{k_{2n+1}}$ of side length $\ell_{k_{2n+1}} \geq 2^{-i_n-2}$. Thus
\begin{equation}
\label{widthestimate_eq}
\beta_K(3Q) = \frac{\omega(3Q)}{\ell(3Q)} \geq \frac{\ell(Q')}{3\ell(Q)} \geq \frac{2^{-i_n-2}}{3\cdot 2^{-i_n}} = \frac{1}{12} \ .
\end{equation}
Therefore with \eqref{dyadic_eq}, \eqref{numbersquaresestimate_eq}, \eqref{widthestimate_eq} and \eqref{preciselength_eq}
\begin{align*}
\sum_{Q \in \cD_{i_n}} \beta_K^2(3Q)\ell(Q) & \gtrsim (\sharp \cD_{i_n}) \cdot 2^{-i_n} \geq 4^{k_{2n+1}-1} \ell_{k_{2n+1}} \gtrsim \tfrac{1}{n} \ .
\end{align*}
Thus
\begin{align*}
\beta^2(K) & \gtrsim \sum_{n \geq 1} \frac{1}{n} = \infty \ .
\end{align*}
\end{proof}

We are now ready to prove Theorem~\ref{main_thm} for $d=1$. 

\begin{proof}[Proof of Theorem~\ref{main_thm} for $K^1$]
The statements on the dimension of $K$ are contained in and Lemma~\ref{lowerbox_lem}. The condition \eqref{lim_cond} is covered by Lemma~\ref{uperbox_lem}. Due to Lemma~\ref{betainfinity_lem}, the set $K$ can not be covered by single a curve of finite length. 

\medskip 

Now assume by contradiction that $K$ can be covered by countably many curves $\Gamma_i$, $i \in \N$, of finite length. Taking the closure we may assume that the sets $\Gamma_i$ are compact. Since $\Gamma_1$ can not cover $K$ there exists a point $p_1 \in K \setminus \Gamma_i$. Since $p_1$ has positive distance to $\Gamma_i$, there exists some $j_1$ and a square $Q_1 \subset \cC_{j_1}$ that contains $p_1$ such that $Q_1 \cap \Gamma_1 = \emptyset$. 

Now we define recursively an increasing sequence $j_1 < j_2 < \cdots$ and $Q_1 \supset Q_2 \supset \cdots$ with $Q_i \in \cC_{j_i}$ and $Q_i \cap (\Gamma_1 \cup \cdots \cup \Gamma_i) = \emptyset$. Assume we have constructed this for the index $i-1$. Now since $K$ is composed of disjoint copies of $Q_{i-1} \cap K$, the set $K \cap Q_{i-1}$ can not be covered by a curve of finite length (otherwise $K$ could be covered by a curve of finite length). Thus as for $i=1$ we find $j_i$, $Q_{j_i} \in \cC_{j_i}$ with $Q_{j_i} \subset Q_{j_{i-1}}$ such that $Q_i \cap \Gamma_i = \emptyset$. Thus $Q_i$ is disjoint from all the curves up to $\Gamma_i$. Since $K$ is complete, the set $\bigcap_{i \geq 1} (K \cap Q_i)$ is nonempty and disjoint from all the curves $\Gamma_i$. 

This proves the theorem for $K^1$.
\end{proof}

\subsection{Construction of $K^d$}

Let us recall that the construction of $K=K^1$ above depends on a certain parameter $\frac{1}{2} < \gamma < 1$. For the construction of $K^d = K'$ we will consider another parameter $\frac{1}{2} < \delta \leq \gamma$ and define $K' \subset \R^2$ as follows. 

The compact set $K' \subset \R^2$ is obtained as $K' = \bigcap_{k \geq 0} K_{k}'$ where $K_{k}$ is the union of a collection $\cC_{k}'$ of $4^k$ disjoint closed squares depending on $\gamma$, $\delta$ and the sequence $(k_n)_n$ already defined above for $K$. Each square in $\cC_{k}'$ has side length $\ell_{k}'$. Again $\cC_{0}' = \{K_0'\}$ where $K_0' \defl [0,1]^2$ and hence $\ell_0' = 1$. In case $k_{2n} \leq k < k_{2n+1}$ for some $n \geq 0$, then each square in $\cC_{k}$ is replaced by the $4$ corner-squares with side length $\ell_{k+1}' = \frac{1}{4^\delta} \ell_{k}'$. Again these squares are disjoint because $\frac{1}{4^\delta} < \frac{1}{2}$ by the restriction on $\delta$. These new squares compose $\cC_{k+1}'$. In case $k_{2n-1} \leq k < k_{2n}$ for some $n \geq 1$, then each square in $\cC_{k}'$ is replaced by the $4$ corner squares of side length $\ell_{k+1}' = \frac{1}{4^{n\delta/\gamma}} \ell_{k}'$ (note that $n\delta/\gamma > \frac{1}{2}$). From this construction and the fact that $\ell_0 = \ell_0'$ it is clear that $\ell_{k}' = \ell_k^{\delta/\gamma}$ for each $k$.

Consider the map $F : K' \to K$ defined as follows: There is an obvious correspondence between squares in $\cC_k'$ and $\cC_k$ for all $k$ and let $F_k : K'_k \to K_k$ be the map that sends each square in $\cC_k'$ to its corresponding square in $\cC_k$ by an affine map. We claim that $F_k$ converges uniformly to a map $F$ on $K'$. To see this we shall check that the sequence $(F_m(x))_m$ is a uniformly Cauchy sequence for $x\in K'$. Indeed, pick a point $x \in Q'$ for some square $Q' \in \cC_k'$ with corresponding square $Q \in \cC_k$, then $F_l(x) \in Q$ for all $l \geq k$. So $|F_l(x) - F_m(x)| \leq \diam(Q) = \ell_k$ for all $l,m \geq k$.

Fix $x,y \in K'$ and let $k$ be the largest integer such that $x,y \in Q'$ for some $Q' \in \cC_k'$ but $x$ and $y$ are in different squares $Q_x'$ and $Q_y'$ of $\cC_{k+1}'$. The corresponding squares are denoted by $Q$, $Q_x$, and $Q_y$ respectively. There is a constant $c \geq 1$ depending only on $\gamma$ and $\delta$ such that
\[
c\dist(Q_x',Q_y') \geq \ell_{k}' \qquad \mbox{and} \qquad c\dist(Q_x,Q_y) \geq \ell_{k} \ .
\]
Thus
\begin{align*}
|F(x)-F(y)| & \leq \diam(Q) = \ell_k = \ell_k'^{\gamma/\delta} \leq c^{\gamma/\delta}\dist(Q_x',Q_y')^{\gamma/\delta} \\
 & \leq c^{\gamma/\delta}|x-y|^{\gamma/\delta} \ .
\end{align*}
Notice also that $F: K' \to K$ is a bijection.  By a similar consideration as above, we obtain that $|F^{-1}(x) - F^{-1}(x)| \leq c^{\delta/\gamma}|x-y|^{\delta/\gamma}$. Thus we get for all $x,y \in K'$,
\begin{equation}
\label{hoelderest_eq}
L^{-1}|x-y|^{\gamma/\delta} \leq |F(x)-F(y)| \leq L^{\delta/\gamma}|x-y|^{\gamma/\delta} \ ,
\end{equation}
for some constant $L \geq 1$ depending only on $\gamma$ and $\delta$. This implies as in Proposition~\ref{holder_prop} that for $0 < \epsilon < 1$
\begin{equation} \label{numberofsquares_eq}
L^{-1} N(K',\epsilon) \leq N(K,\epsilon^{\gamma/\delta}) \leq L N(K',\epsilon) \ , 
\end{equation}
for a constant $L$ independent of $\epsilon$.

\begin{proof}[Proof of Theorem~\ref{main_thm}]
Using the corresponding results for $K$, Proposition~\ref{holder_prop} and \eqref{hoelderest_eq} we obtain
\[
0 \leq \dim_{\rm H}(K') \leq \underline{\dim}_{\rm box}(K') \leq \tfrac{\gamma}{\delta}\underline{\dim}_{\rm box}(K) = 0 \ .
\]
Now \eqref{numberofsquares_eq} implies
\[
\limsup_{\epsilon \downarrow 0} N(K',\epsilon) \epsilon^{\gamma/\delta} \leq \limsup_{\epsilon \downarrow 0} N(K,\epsilon^{\gamma/\delta}) \epsilon^{\gamma/\delta} = 0 \ .
\]
The fact that $K'$ cannot be covered by countably many $\frac{\delta}{\gamma}$-H\"older curves will be deduced from the fact that $K$ cannot be covered by countably many Lipschitz curves. 

In order to see this, assume first there is a $\frac{\delta}{\gamma}$-H\"older curve $\sigma : [0,1] \to \R^2$ that covers $K'$ and set $A \defl \sigma^{-1}(K')$. Then $(F \circ \sigma)|_A : A \to K$ is a Lipschitz cover of $K$. This map can be extended as a Lipschitz map to the whole interval $[0,1]$. But this is not possible by our previous result. 

So there is no such curve that covers $K'$. The analogous argument as for $K$ now shows that $K'$ can't be covered by countably many $\frac{\delta}{\gamma}$-H\"older curves. Note that by Corollary \ref{cover_cor} we have that $\overline{\dim}_{\rm box}(K')\geq \frac{\gamma}{\delta}$ and by construction the fraction $\frac{\gamma}{\delta}$ can take any value in $[1,2)$. This finishes the proof of Theorem~\ref{main_thm}.
\end{proof}



\bigskip
\bigskip
\noindent
University of Bern, Mathematical Institute, Sidlerstrasse 5, 3012 Bern, Switzerland

\bigskip
\noindent
\texttt{zoltan.balogh@math.unibe.ch}

\smallskip
\noindent
\texttt{roger.zuest@math.unibe.ch}

\end{document}